\documentclass[journal]{IEEEtran}
\usepackage{ifpdf}

\usepackage{algorithm}
\usepackage{algpseudocode}
\usepackage{color}
\usepackage{url}
\ifpdf
\usepackage[pdftex]{graphicx}
\else
\usepackage[dvips]{graphicx}
\fi
\usepackage{tabularx}
\DeclareGraphicsExtensions{.eps}
\usepackage[caption=false,font=footnotesize]{subfig}
\usepackage{amsmath,amsthm,amsfonts,amssymb}
\usepackage[noadjust]{cite}
\usepackage{nomencl}
\usepackage{enumerate}
\theoremstyle{definition}
\newtheorem{theorem}{\textbf{Theorem}}
\newtheorem{lemma}{\textbf{Lemma}}
\newcommand{\up}[1]{^\mathrm{#1}}
\newcommand\Tstrut{\rule{0pt}{2.6ex}}         % = `top' strut
\newcommand\Bstrut{\rule[-0.9ex]{0pt}{0pt}}   % = `bottom' strut 

\makeatletter
\newcommand{\subalign}[1]{%
	\vcenter{%
		\Let@ \restore@math@cr \default@tag
		\baselineskip\fontdimen10 \scriptfont\tw@
		\advance\baselineskip\fontdimen12 \scriptfont\tw@
		\lineskip\thr@@\fontdimen8 \scriptfont\thr@@
		\lineskiplimit\lineskip
		\ialign{\hfil$\m@th\scriptstyle##$&$\m@th\scriptstyle{}##$\crcr
			#1\crcr
		}%
	}
}
\makeatother

\begin{document}
	
	\title{Factoring the Cycle Aging Cost of Batteries \\ Participating in Electricity Markets}	
	
	\author{Bolun~Xu,~\IEEEmembership{Student Member,~IEEE,}
		Jinye~Zhao,~\IEEEmembership{Member,~IEEE,}
		Tongxin~Zheng,~\IEEEmembership{Senior Member,~IEEE,}
		Eugene~Litvinov,~\IEEEmembership{Fellow,~IEEE},
		Daniel S.~Kirschen,~\IEEEmembership{Fellow,~IEEE}
		% \vspace{-.7cm}
		\thanks{B.~Xu and D.S.~Kirschen are with the University of Washington, USA (emails: \{xubolun, kirschen\}@uw.edu). }
		\thanks{J.~Zhao, T.~Zheng, and E.~Litvinov are with ISO New England Inc., USA (emails: \{jzhao, tzheng, elitvinov\}@iso-ne.com). }
	}  
	
	\maketitle
	\makenomenclature
	
	\begin{abstract}
		When participating in electricity markets, owners of battery energy storage systems must bid in such a way that their revenues will at least cover their true cost of operation. Since cycle aging of battery cells represents a substantial part of this operating cost, the cost of battery degradation must be factored in these bids. However, existing models of battery degradation either do not fit market clearing software or do not reflect the actual battery aging mechanism. In this paper we model battery cycle aging using a piecewise linear cost function, an approach that provides a close approximation of the cycle aging mechanism of electrochemical batteries and can be incorporated easily into existing market dispatch programs. By defining the marginal aging cost of each battery cycle, we can assess the actual operating profitability of batteries. A case study demonstrates the effectiveness of the proposed model in maximizing the operating profit of a battery energy storage system taking part in the ISO New England energy and reserve markets.
		
	\end{abstract}
	
	\begin{IEEEkeywords}
		Energy storage, battery aging mechanism, arbitrage, ancillary services, economic dispatch
	\end{IEEEkeywords}

	\IEEEpeerreviewmaketitle

	\section{Introduction}
	
	In 2016, about 200~MW of stationary lithium-ion batteries were operating in grid-connected installations worldwide~\cite{eu_bat}, and more deployments have been proposed~\cite{isone_outlook,cal_rm}.
	%For example, in Germany, the power station operator STEAG has announced building six new large-scale 15 MW lithium-ion batteries~\cite{de_bat}. In ISO New England, a total of 94~MW grid-connected battery energy storage capacities have been proposed as of January 2016~\cite{isone_outlook}. While the California Public Utilities Commission has targeted an investor-owned energy storage procurement goal of 1325~MW by the end of 2020~\cite{cal_rm}. 
	To accommodate this rapid growth in installed energy storage capacity, system operators and regulatory authorities are revising operating practices and market rules to take advantage of the value that energy storage can provide to the grid. In particular, the U.S. Federal Energy Regulatory Commission (FERC) has required independent system operators (ISO) and regional transmission organizations (RTO) to propose market rules that  account for the physical and operational characteristics of storage resources~\cite{ferc_rm}. For example, the California Independent System Operator (CAISO) has already designed a market model that supports the participation of energy-limited storage resources and considers constraints on their state of charge (SoC) as well as on their maximum charge and discharge capacity~\cite{caiso_bpm}.

	% has already implemented a new market model tailored for limited energy storage resources~\cite{caiso_ngr}.
	
	As electricity markets evolve to facilitate participation by battery energy storage (BES), owners of these systems must develop bidding strategies which ensure that they will at least recover their operating cost. Battery degradation must be factored in the operating cost of a BES because the life of electrochemical battery cells is very sensitive to the charge and discharge cycles that the battery performs and is thus directly affected by the way it is operated ~\cite{vetter2005ageing,xu2016modeling}.  Existing models of battery degradation either do not fit dispatch calculations, or do not reflect the actual battery degradation mechanism. In particular, traditional generator dispatch models based on heat-rate curves cannot be used to represent the cycle aging characteristic of electrochemical batteries.

	This paper proposes a new and accurate way to model of the cost of battery cycle aging, which can be integrated easily in economic dispatch calculations. The main contributions of this paper can be summarized as follows:
	\begin{itemize}
		\item It proposes a piecewise linear cost function that provides a close approximation of the cost of cycle aging in electrochemical batteries.
		\item  System operators can incorporate this model in market clearing calculations to facilitate the participation of BES in wholesale markets by allowing them to properly reflect their operating cost.
		\item Since this approach defines the marginal cost of battery cycle aging, it makes it possible for BES owners to design market offers and bids that recover at least the cost of battery life lost due to market dispatch.
		\item The effectiveness of the proposed model is demonstrated using a full year of price data from the ISO New England energy markets.
		\item The accuracy of the proposed model in predicting the battery cycle aging cost is demonstrated using an ex-post calculation based on a benchmark model.
		\item The model accuracy increases with the number of linearization segments, and the error compared to the benchmark model approaches zero with sufficient linearization segments.
	\end{itemize}

	% \begin{figure}
	%     \centering
	%     \includegraphics[trim=80mm 35mm 80mm 5mm, clip, width=.6\columnwidth]{BES_cost.eps}
	%     \caption{Design of a typical battery energy storage system}
	%     \label{fig:bes}
	% \end{figure}
	
	Section~II reviews the existing literature on battery cycle aging modeling. Section~III describes the proposed predictive battery cycle aging cost model. Section~IV shows how this model is incorporated in the economic dispatch. Section~V describes and discusses case studies performed using ISO New England market data. Section~VI draws conclusions.

	\section{Literature Review}

	\subsection{Battery Operating Cost}
	
	Previous BES economic studies typically assume that battery cells have a fixed lifetime and do not include the cost of replacing the battery in the BES variable operating and maintenance (O\&M) cost~\cite{kintner2010energy}. The Electricity Storage Handbook from Sandia National Laboratories assumes that a BES performs only one charge/discharge cycle per day, and that the variable O\&M cost of a lithium-ion BES is constant and about 2~\$/MWh~\cite{akhil2013doe}. Similarly, Zakeri \emph{et al.}~\cite{zakeri2015electrical} assume that battery cells in lithium-ion BES are replaced every five years, and assume the same 2~\$/MWh O\&M cost. Other energy storage planning and operation studies also assume that the operating cost of BES is negligible and that they have a fixed lifespan~\cite{pandzic2015near,pozo2014unit,qiu2016stochastic}. 
	These assumptions  are not valid if the BES is cycled multiple times per day because more frequent cycling increases the rate at which battery cells degrade and hasten the time at which they need to be replaced. 
	To secure the battery lifespan, Mohsenian-Rad~\cite{mohsenian2016optimal} caps the number of cycles a battery can operate per day. However, artificially limiting the cycling frequency prevents operators from taking advantage of a BES's operational flexibility and significantly lessens its profitability. To take full advantage of the ability of a BES to take part in energy and ancillary markets, its owner must be able to cycle it multiple times per day and to follow irregular cycles. Under these conditions, its lifetime can no longer be considered as being fixed and its replacement cost can no longer be treated as a capital expense. Instead, the significant part of the battery degradation cost that is driven by cycling should be treated as an operating expense. 
	
	A BES performs temporal arbitrage in an electricity market by charging with energy purchased at a low price, and discharging this stored energy when it can be sold at a higher price. The profitability of this form of arbitrage depends not only on the price difference but also on the cost of the battery cycle aging caused by these charge/discharge cycles. When market prices are stable, the expected arbitrage revenue is small and the BES owner may therefore opt to forgo cycling to prolong the battery lifetime and reduce its cycle aging cost. On the other hand, if the market exhibits frequent large price fluctuations, the BES owner could cycle the BES multiple times a day to maximize its profits. Fig.~\ref{Fig:MP} shows that the price profile in a given market can change significantly from day to day. Although the average market price is higher in Fig.~\ref{Fig:MP_a}, arbitrage is not profitable in this case because the price fluctuations are small, and the aging cost from cycling is likely to be higher than the revenue from arbitrage. On the other hand, a BES owner is likely to perform three arbitrage cycles if the price profile is similar to the one shown on Fig.~\ref{Fig:MP_b}, because the revenue opportunities arising from the large price fluctuations are likely to be larger than the associated cycle aging cost. It is thus crucial to accurately incorporate the cost of cycle aging into the optimal operation of a BES.
	
	\begin{figure}[t]%
		\centering
		\subfloat[An example of stable market prices (Jan 7, 2015).]{
			\includegraphics[trim = 5mm 0mm 10mm 0mm, clip, width = .95\columnwidth]{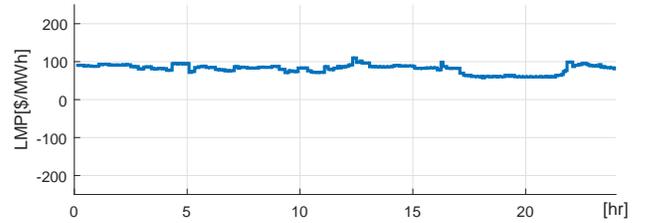}
			\label{Fig:MP_a}%
		}
		\\
		\subfloat[An example of highly variable market prices (Jan 2, 2015).]{
			\includegraphics[trim = 5mm 0mm 10mm 0mm, clip, width = .95\columnwidth]{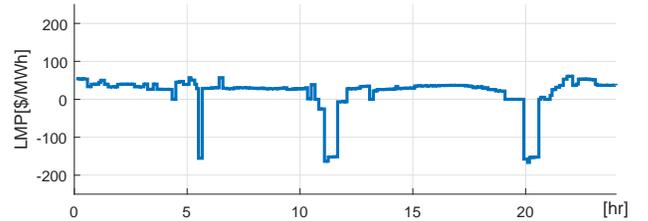}
			\label{Fig:MP_b}%
		}
		\caption{\footnotesize Market price daily variation examples (Source: ISO New England).}%
		\label{Fig:MP}
	\end{figure}

	% and define very small variable operating and maintenance (O\&M) cost for BES~\cite{kintner2010energy,akhil2013doe}. 
	
	% BES units have very small variable operating and maintenance (O\&M) cost because they consume no fuel to generate power. For example, the variable O\&M cost for lithium-ion BES is around 3 to 7\$/MWh~\cite{zakeri2015electrical}. 

	%Cycle aging makes up a substantial part in battery degradation. As an electrochemical battery cycles between charge and discharge, its capacity and performance fades due to various aging mechanisms~\cite{vetter2005ageing}, and thus reduces its lifetime. When the cells in a BES reached their end of life, the BES owner must purchase and install new cells to continue the BES operation, which incurs additional cell replacement cost. Therefore, it is crucial that the cost of cycle aging being incorporated into the BES operating cost. 
	
	\subsection{Electrochemical Battery Degradation Mechanisms}
	
	Electrochemical batteries have limited cycle life~\cite{dunn2011electrical} because of the fading of active materials caused by the charging and discharging cycles. This cycle aging is caused by the growth of cracks in the active materials, a process similar to fatigue in materials subjected to cyclic mechanical loading~\cite{vetter2005ageing,fatemi1998cumulative,li2011crack,wang2011cycle,laresgoiti2015modeling}. Chemists describe this process using partial differential equations~\cite{ramadesigan2012modeling}. These models have good accuracy but cannot be incorporated in dispatch calculations. On the other hand, heuristic battery lifetime assessment models assume that degradation is caused by a set of stress factors, each of which can be represented by a stress model derived from experimental data. The effect of these stress factors varies with the type of battery technology. In this paper, we focus on lithium-ion batteries because they are widely considered as having the highest potential for grid-scale applications. For our purposes, it is convenient to divide these stress factors into two groups depending on whether or not they are directly affected by the way a grid-connected battery is operated: 
	\begin{itemize}
		\item Non-operational factors: ambient temperature, ambient humidity, battery state of life, calendar time~\cite{kassem2012calendar}.
		\item Operational factors: Cycle depth, over charge, over discharge, current rate, and average state of charge (SoC)~\cite{vetter2005ageing}.
	\end{itemize}
	
	\subsubsection{Cycle depth}
	Cycle depth is an important factor in a battery's degradation, and is the most critical component in the BES market dispatch model. A 7~Wh Lithium Nickel Manganese Cobalt Oxide (NMC) battery cell can perform over 50,000 cycles at 10\% cycle depth, yielding a lifetime energy throughput (i.e. the total amount of energy charged and discharged from the cell) of 35~kWh. If the same cell is cycled at 100\% cycle depth, it can only perform 500 cycles, yielding an lifetime energy throughput of only 3.5~kWh~\cite{ecker2014calendar}. This nonlinear aging property with respect to cycle depth is observed in most static electrochemical batteries~\cite{ruetschi2004aging, byrne2012estimating, xu2016modeling, wang2014degradation}. Section~\ref{BES:CA} explains in details our modeling of the cycle depth stress.
	
	\subsubsection{Current rate}
	While high charging and discharging currents accelerate the degradation rate, grid-scale BES normally have capacities greater than 15 minutes. The effect of current rate on degradation is therefore small in energy markets according to results of laboratory tests~\cite{wang2014degradation}. We will therefore not consider the current rate in our model. If necessary, a piecewise linear cost curve can be used to model the current rate stress function as a function of the battery's power output.
	
	\subsubsection{Over charge and over discharge} In addition to the cycle depth effect, extreme SoC levels significantly reduce battery life~\cite{vetter2005ageing}. However, over-charging and over-discharging are avoided by enforcing upper and lower limits on the SoC either in the dispatch or by the battery controller.
	
	\subsubsection{Average state of charge} The average SoC level in each cycle has a highly non-linear but slight effect on the cycle aging rate~\cite{ecker2014calendar, millner2010modeling}. Therefore we do not consider this stress factor in the proposed model.
	
	\subsection{The Rainflow Counting Algorithm}\label{Sec:rf}

	\begin{figure}[t]%
		\centering
		\subfloat[An example of SoC profile.]{
			\includegraphics[trim = 10mm 0mm 10mm 0mm, clip, width = .95\columnwidth]{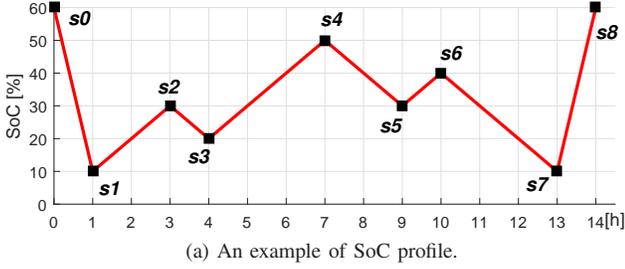}
			\label{Fig:rf1}%
		}
		\\
		\subfloat[The cycle counting result.]{
			\includegraphics[trim = 10mm 0mm 10mm 0mm, clip, width = .95\columnwidth]{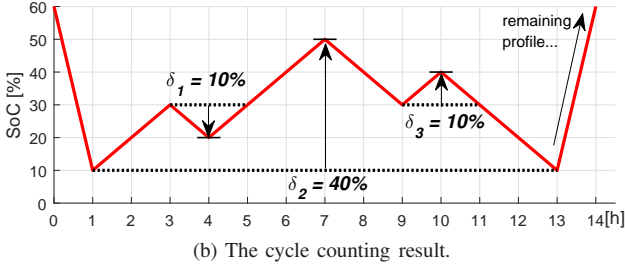}
			\label{Fig:rf2}%
		}
		% 	\\
		% 	\subfloat[Cycle aging function and upper approximations.]{
		% 	    \includegraphics[trim = 5mm 0mm 10mm 0mm, clip, width = .95\columnwidth]{Deg_Cost.eps}
		%     }
		%     \label{Fig:dod}%
		% 	\\
		% 	\subfloat[An exmaple SoC profile.]{
		% 		\includegraphics[trim = 10mm 0mm 10mm 0mm, clip, width = .95\columnwidth]{cycle_ill3.eps}
		% 		\label{Fig:rf1}%
		% 	}
		% 	\subfloat[The cycle counting result.]{
		% 		\includegraphics[trim = 10mm 0mm 10mm 0mm, clip, width = .95\columnwidth]{cycle_ill4.eps}
		% 		\label{Fig:rf2}%
		% 	}
		\caption{\footnotesize Using the rainflow algorithm to identify battery cycle depths.}%
		\label{Fig:rf}
	\end{figure}

	The rainflow counting algorithm is used extensively in materials stress analysis to count cycles and quantify their cumulative impact.  It has also been applied to battery life assessment ~\cite{xu2016modeling, muenzel2015multi}. Given a SoC profile with a series of local extrema (i.e. points where the current direction changed) $s_0$, $s_1$, $\dotsc$, etc, the rainflow method identifies cycles as~\cite{amzallag1994standardization}:
	\begin{enumerate}
		\item Start from the beginning of the profile (as in Fig.~\ref{Fig:rf1}).
		\item Calculate $\Delta s_1 = |s_0-s_1|$, $\Delta s_2 = |s_1-s_2|$, $\Delta s_3 = |s_2-s_3|$. 
		\item If $\Delta s_2\leq \Delta s_1$ and $\Delta s_2 \leq \Delta s_3$, then a full cycle of depth $\Delta s_2$ associated with $s_1$ and $s_2$ has been identified. Remove $s_1$ and $s_2 $ from the profile, and repeat the identification using points $s_0$, $s_1$, $s_4$, $s_5$...
		\item If a cycle has not been identified, shift the identification forward and repeat the identification using points $s_1$, $s_2$, $s_3$, $s_4$...
		\item The identification is repeated until no more full cycles can be identified throughout the remaining profile.
	\end{enumerate}
	The remainder of the profile is called the rainflow residue and contains only half cycles~\cite{marsh2016review}. A half cycle links each pair of adjoining local extrema in the rainflow residue profile. A half cycle with decreasing SoC is a discharging half cycle, while  a half cycle with increasing SoC is a charging half cycle. For example, the SoC profile shown on Fig.~\ref{Fig:rf2} has two full cycles of depth 10\% and one full cycle of depth 40\%, as well as a discharging half cycle of depth 50\% and charging half cycle of depth 50\%.

	% Given a time series $\sigma_t$ of a battery's relative state of charge (SoC), this algorithm identifies the depth of all cycles contained in this series: 
	% \begin{align}
	%     \delta_i = \Gamma(\sigma_{1},\dotsc,\sigma_{T})\,,
	%     \label{Eq:rf1}
	% \end{align}
	% where $\Gamma$ represents the rainflow algorithm, and $\delta_i$ is the depth of the $i$th cycle. 
	
	The rainflow algorithm does not have an analytical mathematical expression~\cite{benasciutti2005spectral} and cannot be integrated directly within an optimization problem. Nevertheless, several efforts have been made to optimize battery operation by simplifying the rainflow algorithm. Abdulla \emph{et al.}~\cite{abdulla2016optimal} and Tran \emph{et al.}~\cite{tran2013energy} simplify the cycle depth as the BES energy output within each control time interval. Koller \emph{et al.}~\cite{koller2013defining} define  a cycle as the period between battery charging and discharging transitions. These model simplifications enable the incorporation of cycle depth in the optimization of BES operation, but introduce additional errors in the degradation model. He \emph{et al.}~\cite{he2015optimal} decompose the battery degradation model and optimize BES market offers iteratively. This method yields more accurate dispatch results, but is too complicated to be incorporated in an economic dispatch calculation.
	
	We will use the rainflow algorithm as the basis for an ex-post benchmark method for assessing battery cycle life. In this model, the total life lost $L$ from a SoC profile is assumed to be the sum of the life loss from all $I$ number of cycles identified by the rainflow algorithm. If the life loss from a cycle of depth $\delta$ is given by a cycle depth stress function $\Phi(\delta)$ of polynomial form, we have:
	\begin{align}
	L = \textstyle \sum_{i=1}^{I} \Phi(\delta_i)\,.
	\label{Eq:rf1}
	\end{align}
	% where $\Phi(\delta)$ is a polynomial function. 
	
	% with coefficients $\alpha$ and $m$ fitted from experimental data: 
	%  \begin{align}
	%      \Phi(\delta) = \alpha \delta^{m}\,.
	%      \label{Eq:rf2}
	%  \end{align}
	
	\section{Marginal Cost of Battery Cycling}\label{BES:CA}
	
	\begin{figure}
		\centering
		\includegraphics[trim = 5mm 0mm 10mm 0mm, clip, width = .95\columnwidth]{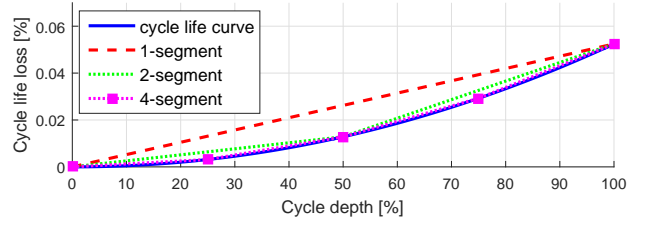}
		\caption{Upper-approximation to the cycle depth aging stress function. 
			%As the number of segments increase, the piecewise linear function provides a better approximation of the nonlinear cycle life loss.
		}
		\label{Fig:dod}%
	\end{figure}
	
	% \begin{figure*}[t]%
	% 	\centering
	% 	\vspace{-8mm}
	%     \subfloat[Power profile.]{
	% 		\includegraphics[trim = 5mm 0mm 10mm 3mm, clip, width = .95\columnwidth]{CA_p.eps}
	% 		\label{Fig:CA_p}%
	% 	}
	% 	\subfloat[SoC profile.]{
	% 		\includegraphics[trim = 5mm 0mm 10mm 3mm, clip, width = .95\columnwidth]{CA_soc.eps}
	% 		\label{Fig:CA_soc}%
	% 	}
	% 	\\
	% 	\subfloat[Marginal cost of cycle aging.]{
	% 		\includegraphics[trim = 5mm 0mm 10mm 3mm, clip, width = .95\columnwidth]{CA_mc.eps}
	% 		\label{Fig:CA_mc}%
	% 	}
	% 	\subfloat[Accumulated cost of cycle aging.]{
	% 		\includegraphics[trim = 5mm 0mm 10mm 3mm, clip, width = .95\columnwidth]{CA_cost.eps}
	% 		\label{Fig:CA_cost}%
	% 	}
	% 	\caption{Effect of shorter (Case 1) and longer (Case 2) discharge periods on battery cycle aging. 
	% 	% As shown in figure (a) and (b), Case 1 produces shallower SoC cycles compared to the deeper SoC cycles of Case 2. Figure (c) shows that in Case 1 the marginal cycle aging cost stays at 10~\$/MWh, while for Case 2 this marginal cost continues to increase with the depth of the cycle. The stepwise nature of this increase is a consequence of the piecewise linear nature of the marginal cost curve. As shown in figure (d), the total cost of cycle aging is the integral of the marginal cycle aging cost.
	% 	}%
	% 	\label{Fig:CA}
	% \end{figure*}
	
	In order to participate fully in electricity markets, owners of batteries must be able to submit offers and bids that reflect their marginal operating cost. As we argued above, this marginal cost curve should reflect the cost of battery degradation caused by each cycle. In order to keep the model simple, and to obtain a cost function similar to those used in existing market dispatch programs, we assume that battery cycle aging only occurs during the discharge stage of a cycle, so that a discharging half cycle causes the same cycle aging as a full cycle of the same depth, while a charging half cycles causes no cycle aging. This is a reasonable assumption because the amounts of energy charged and discharged from a battery are almost identical when assessed on a daily basis.
	
	During a cycle, if the BES is discharged from a starting SoC $e\up{up}$ to an end SoC $e\up{dn}$ and later charged back (or vice-versa), the depth of this cycle is the relative SoC difference $(e\up{up}-e\up{dn})/E\up{rate}$, where $E\up{rate}$ is the energy capacity of the BES.
	Let a battery be discharged from a cycle depth $\delta_{t-1}$ at time interval $t-1$. This battery's cycle depth at time $t$ can be calculated from its output power $g_t$ over time (assuming the time interval duration is one hour):
	\begin{align}
	\delta_t = \frac{1}{\eta\up{dis}E\up{rate}}g_t + \delta_{t-1}\,,
	\label{ES:dod}
	\end{align}
	where $\eta\up{dis}$ is the BES discharge efficiency, and $g_t$ has non-negative values because we ignore charging for now. The incremental aging resulting from this cycle is $\Phi(\delta_t)$, and the marginal cycle aging can be calculated by taking the derivative of $\Phi(\delta_t)$ with respect to $g_t$ and substituting from \eqref{ES:dod}:
	\begin{align}
	\frac{\partial \Phi(\delta_i)}{\partial g_t} = \frac{d\Phi(\delta_i)}{d\delta_i}\frac{\partial \delta_i}{\partial g_t} = \frac{1}{\eta\up{dis}E\up{rate}}\frac{d\Phi(\delta_i)}{d\delta_i}\,,
	\label{ES:inc}
	\end{align}
	
	To define the marginal cost of cycle aging, we prorate the battery cell replacement cost $R$ (\$) to the marginal cycle aging, and construct a piecewise linear upper-approximation function $c$. This function consists of $J$ segments that evenly divide the cycle depth range (from 0 to 100\%) 
	% \begin{equation}
	%     c\up{}(\delta_t) =
	%     \begin{cases}
	%     c\up{}_1 = R\frac{\eta\up{dis}}{E\up{rate}} J\big[\Phi(\frac{1}{J})\big] & \text{if } \delta_t \in [0, \frac{1}{J}) \\
	%     \vdots & \\
	%     c\up{}_j = R\frac{\eta\up{dis}}{E\up{rate}}J\big[\Phi(\frac{j}{J})-\Phi(\frac{j-1}{J})\big] & \text{if } 
	%     \delta_t \in [\frac{j-1}{J},  \frac{j}{J}) \\
	%     \vdots & \\
	%     c\up{}_{J} = R\frac{\eta\up{dis}}{E\up{rate}}J\big[\Phi(1)-\Phi(\frac{J-1}{J})\big] & \text{if } 
	%     \delta_t \in [\frac{J-1}{J}, 1]
	%     \end{cases},
	%     \label{ES:ca_pl}
	% \end{equation}
	\begin{equation}
	c\up{}(\delta_t) =
	\begin{cases}
	c\up{}_1  & \text{if } \delta_t \in [0, \frac{1}{J}) \\
	\vdots & \\
	c\up{}_j & \text{if } 
	\delta_t \in [\frac{j-1}{J},  \frac{j}{J}) \\
	\vdots & \\
	c\up{}_{J} & \text{if } 
	\delta_t \in [\frac{J-1}{J}, 1]
	\end{cases}\,,
	\label{ES:ca_pl}
	\end{equation}
	where
	\begin{equation}
	c\up{}_j = \frac{R}{\eta\up{dis}E\up{rate}}J\big[\Phi(\frac{j}{J})-\Phi(\frac{j-1}{J})\big]\,,
	\end{equation}
	and $\delta_t$ is the cycle depth of the battery at time $t$. Fig.~\ref{Fig:dod} illustrates the cycle depth stress function and its piecewise linearization with different numbers of segments.

	% Fig.~\ref{Fig:CA} illustrates how the cycling pattern affects the cycle aging cost. Each new discharge cycle starts on the first cost segment, i.e. from a zero depth of discharge. As Fig.~\ref{Fig:CA} shows, the marginal cost of discharge increases as the discharges moves into segments corresponding to a deeper discharge. Since Case 1 involves shorter periods of discharge, the SoC cycles are shallower and thus have a lower cost than the deeper cycles that result from the longer discharge periods of Case 2. The effect of charging on cycle life is taken into account in the optimization model described in the next section.

	\section{Optimizing the BES Dispatch}\label{Sec:opt}
	
	Having established a marginal cost function for a BES, we are now able to optimize how it should be dispatched assuming that it acts as a price-taker on the basis of perfect forecasts of the market prices for energy and reserve. A formal description of this optimization requires the definitions of the following parameters:
	\begin{itemize}
		\item $T$: Number of time intervals in the optimization horizon, indexed by $t$
		\item $J$: Number of segments in the cycle aging cost function, indexed by $j$
		\item $M$: Duration of a market dispatch time interval
		\item $S$: Sustainability time requirement for reserve provision
		\item $E\up{0}$: Initial amount of energy stored in the BES
		\item $E\up{final}$: Amount of energy that must be stored at the end of the optimization horizon
		\item $E\up{min}$ and $E\up{max}$: Minimum and maximum energy stored in the BES
		\item $D$, $G$: Discharging and charging power ratings
		\item $c_j$: Marginal aging cost of cycle depth segment $j$
		\item $\overline{e}\up{}_j$: Maximum amount of energy that can be stored in cycle depth segment $j$
		\item $\overline{e}\up{0}_j$: Initial amount of energy of cycle depth segment $j$
		\item $\eta\up{ch}$, $\eta\up{dis}$: Charge and discharge efficiencies 
		\item $\lambda\up{e}_t$, $\lambda\up{q}_t$: Forecasts of the energy and reserve prices at  $t$
	\end{itemize}
	This optimization uses the following decision variables:
	\begin{itemize}
		\item $p\up{ch}_{t,j}$, $p\up{dis}_{t,j}$: Charge and discharge power for cycle depth segment $j$ at time $t$
		\item $e\up{}_{t,j}$: Energy stored in marginal cost segment $j$ at time $t$
		\item $d_t$, $g_t$: Charging and discharging power at time $t$
		\item $d\up{q}_t$, $g\up{q}_t$: BES baseline charging and discharging power at time $t$ for reserve provision
		\item $q_t$: Reserve capacity provided by the BES at time $t$
		\item $v_t$: Operating mode of the BES: if at time $t$ the BES is charging then $v_t=0$; if it is discharging then $v_t=1$. If the BES is idling, this variable can take either value. If some sufficient conditions are satisfied, for example the market clearing prices should not  be negative, the binary variable $v_t$ can be relaxed~\cite{li2016sufficient}
		\item $u_t$: If at time $t$ the BES provides reserve then $u_t=1$, else $u_t=0$
	\end{itemize}
	
	The objective of this optimization is to maximize the operating profit $\Omega$ of the BES. This profit is defined as the difference between the revenues from the energy and reserve markets and the cycle aging cost $C$
	\begin{align}
	\max_{\mathbf{p,g,d,q}} \Omega := \textstyle\sum_{t=1}^{T}M\Big[ \lambda\up{e}_t(g_t-d_t) + \lambda\up{q}_t q_t\Big] - C\up{}\,.
	\label{ES:obj}
	\end{align}
	
	Depending on the discharge power, the depth of discharge during each time interval extends over one or more segments. 
	
	To model the cycle depth in multi-interval operation, we assign a charge power component $p\up{ch}_{t,j}$ and an energy level $e\up{}_{t,j}$ to each cycle depth segment, so that we can track the energy level of each segment independently and identify the current cycle depth. For example, assume we divide the cycle depth of a 1~MWh BES into 10 segments of 0.1~MWh. If a cycle of 10\% depth starts with a discharge, as between $s_2$ and $s_3$ in Fig.~\ref{Fig:rf1}, the BES must have previously undergone a charge event which stored more than 0.1~MWh according to the definition from the rainflow method. Because the marginal cost curve is convex, the BES always discharges from the cheapest (shallowest) available cycle depth segment towards the more expensive (deeper) segments. 
	% Thus, the first depth segment must have been fully charged at the beginning of this cycle, and all energy will be discharged from segment $e_1$ associated with marginal cost $c_1$ during this cycle. Because the charge and discharge components of a cycle are symmetric, at the end of a cycle, the energy level of each segment $e_j$ is restored to its state at the beginning of the cycle, hence the rest of the operation is not affected. 
	So that the proposed model provides a close approximation to the rainflow cycle counting algorithm, detailed proofs and a numerical example are included in the appendix.

	The cycle aging cost $C$ is  the sum of the cycle aging costs associated with each segment over the horizon: 
	\begin{equation}
	C\up{} = \textstyle\sum_{t=1}^{T}\sum_{j=1}^{J}Mc\up{}_jp\up{dis}_{t,j}\,.
	\label{ES:cost}
	\end{equation}
	
	This optimization is subject to the following constraints
	\begin{align}
	d_t &= \textstyle\sum_{j=1}^{J}p\up{ch}_{t,j}% /\eta\up{ch}
	\label{Eq:CP_1}\\
	g_t &= \textstyle\sum_{j=1}^{J}p\up{dis}_{t,j}% \eta\up{dis}
	\label{Eq:CP_2}\\
	d_t &\leq D(1-v_t)
	\label{Eq:CP_3}\\
	g_t &\leq Gv_t%\,,
	\label{Eq:CP_4}\\
	%\end{align}
	%\begin{align}
	e\up{}_{t,j}-e\up{}_{t-1,j} &= M(p\up{ch}_{t,j}\eta\up{ch}- p\up{dis}_{t,j}/\eta\up{dis})
	\label{Eq:CE_1}\\
	e\up{}_{t,j} &\leq \overline{e}\up{}_j
	\label{Eq:CE_2}\\
	E\up{min} \leq \textstyle \textstyle\sum_{j=1}^{J}e\up{}_{t,j} &\leq E\up{max}%\,,
	\label{Eq:CE_4}\\
	e\up{}_{1,j} &= e\up{0}_j
	\label{Eq:CE_5}\\
	\textstyle\sum_{j=1}^{J}e\up{}_{T,j} &\geq E\up{final}\,,
	\label{Eq:CE_3}
	\end{align}
	
	% \begin{figure}
	%     \centering
	%     \includegraphics[trim = 5mm 0mm 10mm 0mm, clip, width = .95\columnwidth]{cycle_illustration.eps}
	%     \caption{Example of a mixed-cycle operation with a deeper cycle A interrupted by a shallower cycle B.
	%     }
	%     \label{Fig:cycle}%
	% \end{figure}
	
	Eq.~\eqref{Eq:CP_1} states that the BES charging power drawn from the grid is the sum of the charging powers associated with each cycle depth segment. Eq.~\eqref{Eq:CP_2} is the equivalent for the discharging power. Eqs.~\eqref{Eq:CP_3}--\eqref{Eq:CP_4} enforce the BES power rating, with the binary variable $v_t$ preventing simultaneous charging and discharging~\cite{go2016assessing}. Eq.~\eqref{Eq:CE_1} tracks the evolution of the energy stored in each cycle depth segment, factoring in the charging and discharging efficiency. Eq.~\eqref{Eq:CE_2} enforces the upper limit on each segment while Eq.~\eqref{Eq:CE_4} enforces the minimum and maximum SoC of the BES. Eq.~\eqref{Eq:CE_5} sets the initial energy level in each cycle depth segment, and the final storage energy level is enforced by Eq.~\eqref{Eq:CE_3}. 
	
	Because the revenues that a BES collects from providing reserve capacity are co-optimized with the revenues from the energy market, it must abide by the  requirements that the North American Electric Reliability Corporation (NERC) imposes on the provision of reserve by energy storage. In particular, NERC requires that a BES must have enough energy stored to sustain its committed reserve capacity and baseline power dispatch for at least one hour~\cite{nerc_reserve}. This requirement is automatically satisfied when market resources are cleared over an hourly interval, as is the case for the ISO New England day-ahead market. If the dispatch interval is shorter than one hour (e.g.for the five-minute ISO New England real-time market), this one-hour sustainability requirement has significant implications on the dispatch of a BES because of the interactions between its power and energy capacities. For example, let us consider a 36~MW BES with 3~MWh of stored energy. If this BES is not scheduled to provide reserve, it can dispatch up to 36~MW of generation for the next 5-minute market period. On the other hand, if it is scheduled to provide 1~MW of reserve, its generation capacity is also constrained by the one hour sustainability requirement, therefore it can only provide up to 2~MW baseline generation for the next 5-minute market period.
	\begin{align}
	0\leq d_t - d\up{q}_t &\leq D(1-u_t)
	\label{Eq:CR_1}\\
	0\leq g_t - g\up{q}_t &\leq G(1-u_t)
	\label{Eq:CR_2}\\
	d\up{q}_t &\leq Du_t
	\label{Eq:CR_3}\\
	g\up{q}_t &\leq Gu_t
	\label{Eq:CR_4}\\
	g\up{q}_t + q_t - d\up{q}_t &\leq Gu_t
	\label{Eq:CR_5}\\
	q_t & \geq \varepsilon u_t
	\label{Eq:CR_6}\\
	S(g\up{q}_t + q_t - d\up{q}_t) &\leq \textstyle\sum_{j=1}^{J}\overline{e}\up{}_j\,,
	\label{Eq:CR_7}
	\end{align}
	Equations \eqref{Eq:CR_1}--\eqref{Eq:CR_7}, enforce the constraints related to the provision of reserve by a BES. In particular, Eq.~\eqref{Eq:CR_7} enforces the one-hour reserve sustainability requirement. Depending on the requirements of the reserve market, the binary variable $u_t$ and constraints \eqref{Eq:CR_1}--\eqref{Eq:CR_7} can be simplified or relaxed.

	The optimization model described above can be used by the BES owner to design bids and offers or self-schedule based on price forecasts. The ISO can also incorporate this model into the market clearing program to better incorporate the aging characteristic of BES. In this case, the cycle aging cost function should be included in the welfare maximization while constraints \eqref{Eq:CP_1}~-~\eqref{Eq:CR_7}  should be added to the market clearing program constraints. A BES owner should include cycle aging parameters $c_j$ and $\overline{e}_j$ in its market offers, and parameters $D$, $G$, $E\up{min}$, $E\up{final}$, $\eta\up{ch}$, $\eta\up{dis}$ for ISO to manage its SoC and its upper/lower charge limits.
	
	% In this case, the cycle aging cost function should be included in the welfare maximization while constraints \eqref{Eq:CP_1}~-~\eqref{Eq:CR_7}  should be added to the market clearing program constraints. 

	\section{Case Study}
	
	The proposed model has been tested using data from ISO New England to demonstrate that it improves the profitability and longevity of a BES participating in this market. All simulations were carried out in GAMS using CPLEX solver~\cite{GAMS}, and the optimization period is 24 hours for all simulations.
	
	\subsection{BES Test Parameters}\label{Sec:CS:BES}
	The BES simulated in this case study has the following parameters:
	\begin{itemize}
		\item Charging and discharging power rating: 20 MW
		\item Energy capacity: 12.5 MWh
		%    \item Power to energy ratio: 1.6 $\mathrm{h^{-1}}$
		\item Charging and discharging efficiency: 95\%
		\item Maximum state of charge: 95\%
		\item Minimum state of charge: 15\%
		\item Battery cycle life: 3000 cycles at at 80\% cycle depth
		\item Battery shelf life: 10 years
		\item Cell temperature: maintained at 25$^\circ C$
		\item Battery pack replacement cost: 300,000~\$/MWh
		\item $\mathrm{Li(NiMnCo)O_2}$-based 18650 lithium-ion battery cells
		%    \item BES PCS investment cost: 500,000~\$/MW
		%    \item BES BMS investment cost: 400,000~\$/MWh
	\end{itemize}
	These cells have a near-quadratic cycle depth stress function~\cite{laresgoiti2015modeling}:
	\begin{align}
	\Phi(\delta) = (5.24\text{E-4})\delta\up{2.03}\,.
	\label{Eq:DoD}
	\end{align}
	Fig.~\ref{Fig:dod} shows this stress function along with several possible piecewise linearizations. We assume that all battery cells are identically manufactured, that the battery management system is ideal, and thus that all battery cells in the BES age at the same rate.
	Since the BES dispatch is performed based on perfectly accurate price forecasts, our results provide an upper bound of its profitability in this market.
	
	\begin{figure*}%
		\centering
		\vspace{-5mm}
		\subfloat[Locational marginal price in the day-ahead market (DAM), in an hourly and a 5-minute real-time market (RTM).]{
			\includegraphics[trim = 25mm 0mm 20mm 3mm, clip, width = 2\columnwidth]{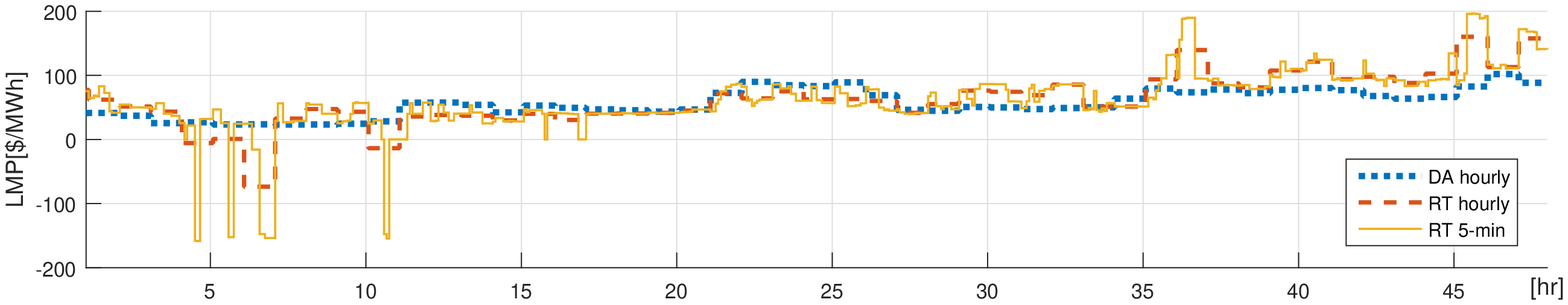}
			\label{Fig:dispatch_price}%
		}
		\\
		\subfloat[BES SoC profile in RTM with 5-minute settlement.]{
			\includegraphics[trim = 25mm 0mm 20mm 0mm, clip, width = 2\columnwidth]{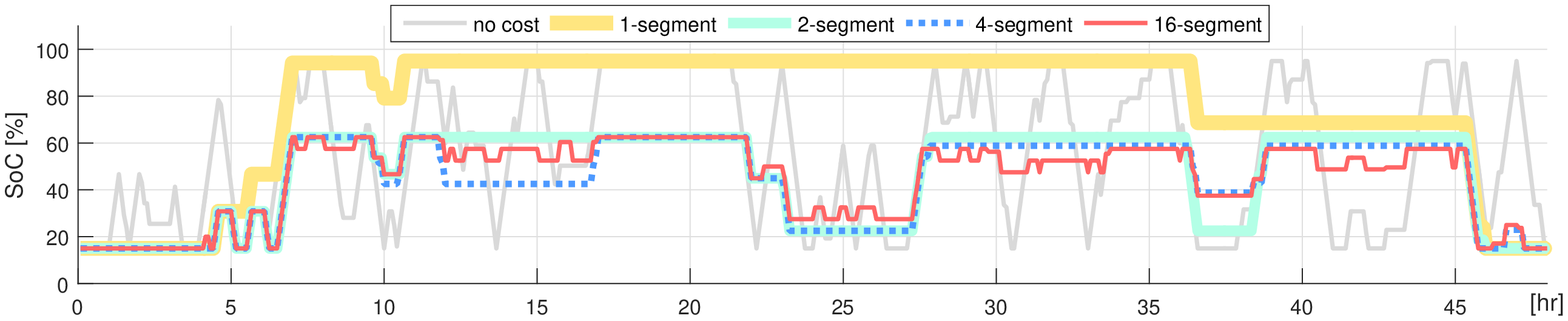}
			\label{Fig:dispatch_soc}%
		}
		\\
		\subfloat[BES output power profile in RTM with 5-minute settlement.]{
			\includegraphics[trim = 25mm 0mm 20mm 0mm, clip, width = 2\columnwidth]{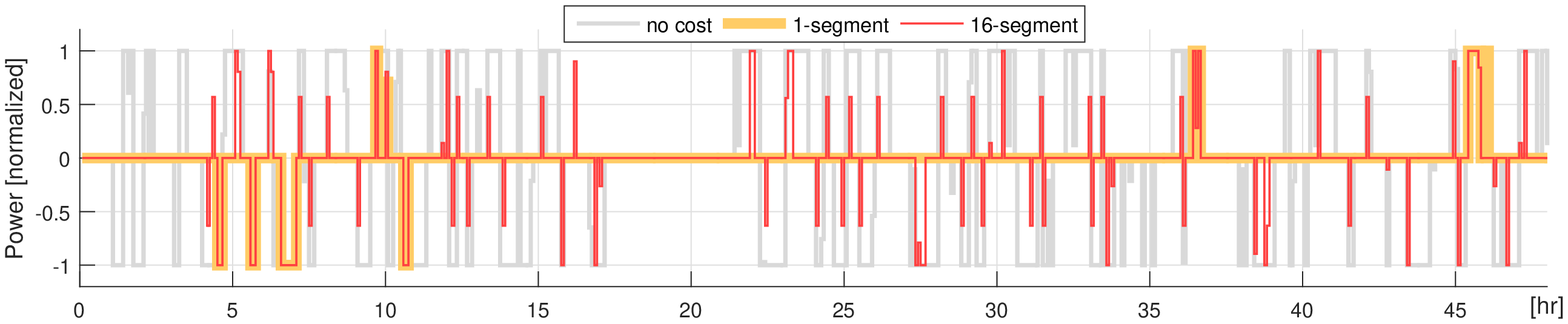}
			\label{Fig:dispatch_p}%
		}
		\caption{BES dispatch for different cycle aging cost models (ISO New England SE-MASS Zone, Jan 5th \& 6th, 2015).}%
		\label{Fig:comp_block}
	\end{figure*}
	
	\subsection{Market Data}\label{Sec:CS:DS}
	
	BES dispatch simulations were performed using zonal price for Southeast Massachusetts (SE-MASS) region of ISO New England market price data for 2015~\cite{iso_express} because energy storage has the highest profit potential in this price zone~\cite{yury2016ensuring}. Three market scenarios were simulated:
	\begin{itemize}
		\item \emph{Day-Ahead Market (DAM)}: Generations and demands are settled using hourly day-ahead prices in this energy market. The DAM does not clear operating reserve capacities. DAM is a purely financial market, and is used in this study to demonstrate the BES dispatch under stable energy prices.
		\item \emph{Real-Time Market (RTM) with 1-hour settlement period}: The real-time energy market clears every five minutes and generates  5-minute real-time energy and reserve prices. Generations, demands, and reserves are settled hourly using an average of these 5-minute prices. The reserve sustainability requirement is one hour.
		\item \emph{RTM with 5-minute settlement periods}: ISO New England plans to launch the 5-minute subhourly settlement on March 1,~2017~\cite{isone_rte}. The reserve sustainability requirement remains one hour.
	\end{itemize}
	Fig.~\ref{Fig:dispatch_price} compares the energy prices in these different markets and shows that the 5-minute real-time prices fluctuate the most, while the day-ahead prices are more stable than real-time prices.

	\subsection{Accuracy of the Predictive Aging Model}
	Fig.~\ref{Fig:dispatch_soc} and \ref{Fig:dispatch_p} compare the BES dispatches for piecewise-linear cycle aging cost functions with different numbers of cycle depth segments. A cost curve with more segments is a closer approximation of the actual cycle aging function. The price signal for these examples is the 5-minute RTM price curve shown in Fig.~\ref{Fig:dispatch_price}. Fig.~\ref{Fig:dispatch_soc} shows the SoC profile while Fig.~\ref{Fig:dispatch_p} shows the corresponding output power profile, where positive values correspond to discharging periods, and negative values to charging periods.  
	
	The gray curve in Fig.~\ref{Fig:comp_block} shows the dispatch of the BES assuming zero operating cost. This is the most aggressive dispatch, and the BES assigns full power to arbitrage as long as there are price fluctuations, regardless of the magnitude of the price differences. Fig.~\ref{Fig:dispatch_p} shows that the BES frequently switches between charging and discharging, and Fig.~\ref{Fig:dispatch_soc} that it ramps aggressively. This dispatch maximizes the market revenue for the BES, but not the maximum lifetime profit, because the arbitrage decisions ignore the cost of cycle aging. We will show in Section~\ref{Sec:CS_MP} that this dispatch actually results in negative profits for all market scenarios. 
	
	\begin{table*}[t]
		\centering
		\vspace{-7mm}
		\caption{Dispatch of a 20MW~/~12.5MWH BES in ISO-NE Energy Markets  (full-year 2015).}
		\begin{tabular}{l || c c c || c c c || c c c}
			\hline
			\hline
			Market & \multicolumn{3}{c||}{DAM} & \multicolumn{3}{c||}{RTM with hourly settlement} & \multicolumn{3}{c}{RTM with 5-minute settlement} \Tstrut\Bstrut\\
			\hline
			Cycle aging cost model            & no cost & 1-seg.  & 16-seg. & no cost & 1-seg.  & 16-seg. & no cost & 1-seg.  & 16-seg. \Tstrut\Bstrut\\
			\hline
			Annual market revenue [k\$]   & 138.8 & 0 & 21.3 & 382.5 & 197.5 & 212.5 & 789.3 & 303.8 & 372.3  \Tstrut\Bstrut\\
			\hline
			Revenue from reserve [\%]        & \multicolumn{3}{c||}{No price for reserve in DAM}& 29.6 & 74.1 & 73.6 & 13.8 & 34.9 & 29.8  \Tstrut\Bstrut\\
			\hline
			%         Arbitrage percentage [\%]      & 100 & 100 & 100 & 70.4 & 25.9 & 26.4 & 86.2 & 65.1 & 70.2  \Tstrut\Bstrut\\
			%         \hline
			Annual life loss from cycling [\%] & 24.4 & 0 & 0.3 & 43.6 & 1.0 & 1.1 & 77.0 & 2.2  & 2.6 \Tstrut\Bstrut\\
			\hline
			Annual prorated cycle aging cost [k\$]        & 913.8 & 0 & 11.3 & 1626.3 & 36.3 & 38.8 & 2887.5 & 81.3 & 96.3 \Tstrut\Bstrut\\
			\hline
			Annual prorated profit [k\$]     & -775.0 & 0 &10 & -1243.8 & 161.3 & 173.8 & -2101.3 & 222.5 & 276.3 \Tstrut\Bstrut\\
			\hline
			Profit from reserve [\%]     & \multicolumn{3}{c||}{No price for reserve in DAM} & - & 90.7 & 90.0 & - & 47.7 & 40.2 \Tstrut\Bstrut\\
			\hline
			Battery life expectancy [year]  & 2.9 & 10.0 & 9.7 & 1.9 & 9.1 & 9.1 & 1.1 & 8.2 & 8.0 \Tstrut\Bstrut\\
			\hline
			\hline
		\end{tabular}
		\label{tab:bes_dispatch}
	\end{table*}

	The yellow curve in Fig.~\ref{Fig:comp_block} illustrates the dispatch of the BES when the cycle aging cost curve is approximated by a single cycle depth segment. In this case, the marginal cost of cycle aging is constant and, as shown in Fig.~\ref{Fig:dod}, it overestimates the marginal cost of aging over a wide range of cycle depths. Therefore, this dispatch yields the most conservative arbitrage response, and the BES remains idle unless price deviations are very large, as demonstrated in Fig.~\ref{Fig:dispatch_p}. Consequently, the BES collects the smallest market revenues, but the BES never loses money from market dispatch because the actual cycle aging is always smaller than the value predicted by the model.
	
	As the number of segments increases, the BES dispatch becomes more sensitive to the magnitude of the price fluctuations, and a tighter correlation can be observed between the market price in Fig.~\ref{Fig:dispatch_price} and the BES SoC in Fig.~\ref{Fig:dispatch_soc}. The red curve shows the dispatch of the BES using a 16-segment linearization of the cycle aging cost curve. When small price fluctuation occurs, the BES only dispatches at a fraction of its power rating, even though it has sufficient energy capacity. This ensures that the marginal cost of cycle aging does not exceed the marginal market arbitrage income.

	Besides considering the impact of the piecewise linearization on the BES dispatch, it is also important to compare the cycle aging cost used by the predictive model incorporated in the dispatch calculation with an ex-post calculation of this cost using the benchmark rainflow-counting algorithm. Using the $\hat{e}_{t,j}$ calculated using the optimal dispatch model \eqref{ES:obj}, we generate a percentage SoC series:
	\begin{align}
	\sigma_t = \textstyle \sum_{j=1}^J {\hat{e}_{t,j}}/{E\up{rate}}\,,
	\end{align}
	This SoC series is fed into the rainflow method as described in Section~\ref{Sec:rf}, and the cycle life loss $L$ is calculated as in \eqref{Eq:rf1} with the cycle stress function \eqref{Eq:DoD}. The relative error $\epsilon$ on the cycle aging cost is calculated as:
	\begin{align}
	\epsilon = {|\hat{C} - RL|}/({RL})\,,
	\end{align}
	where $\hat{C}$ is the cycle aging cost from \eqref{ES:cost}. Fig.~\ref{Fig:cost_err} shows the difference between the predicted and ex-post calculations for the simulations based on the RTM with a 5-minute settlement. As the number of segments increases to 16, the error becomes negligible.

	\begin{figure}
		\centering
		\includegraphics[trim = 5mm 0mm 10mm 0mm, clip, width = .95\columnwidth]{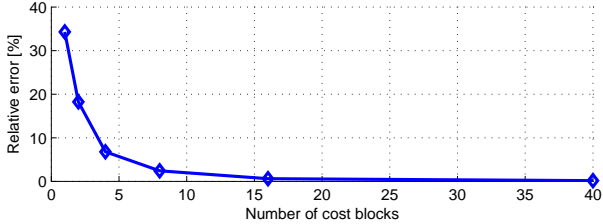}
		\caption{Difference between the cycle aging cost calculated using the predictive model and an ex-post calculation using the benchmark rainflow method for a full-year 5-minute RTM dispatch simulation.}
		\label{Fig:cost_err}%
	\end{figure}
	
	\subsection{BES Market Profitability Analysis}~\label{Sec:CS_MP}
	
	Table~\ref{tab:bes_dispatch} summarizes the economics of BES operation under the three markets described in Section \ref{Sec:CS:DS} and for three cycle aging cost models: \emph{no operating cost}; \emph{single segment cycle aging cost}; and \emph{16-segment cycle aging cost}. The market revenue, profit, and battery life expectancy calculations are based on dispatch simulations using market data spanning all of 2015. On the fifth row, the life loss due to market dispatch is calculated using the benchmark cycle life loss model of Eqs. \eqref{Eq:rf1}, and \eqref{Eq:DoD}. In the sixth row, we calculate the cycle aging cost by prorating the battery cell replacement cost to the dispatch life loss. In the seventh row, the cost of cycle aging is subtracted from the market revenue to calculate the operating profit. In the last row, we estimate the battery cell life expectancy assuming the BES repeats the same operating pattern in future years. The life estimation $L\up{exp}$  includes shelf (calendar) aging and cycle aging
	\begin{align}
	L\up{exp} = ({100\%})/({\Delta L\up{cal} + \Delta L\up{cycle}})\,,
	\end{align}
	where $\Delta L\up{cal}$ is the 10\% annual self life loss as listed in Section~\ref{Sec:CS:BES}, and $\Delta L\up{cycle}$ is the annual life loss due to cycle aging as shown in the sixth row in Table~\ref{tab:bes_dispatch}.

	The 16-segment model generates the largest profit in all market scenarios. Compared to the 16-segment model, the no cost model results in a more aggressive operation of the BES, while the 1-segment model is more conservative. Because the no-cost model encourages arbitrage in response to all price differences, it results in a very large negative profit and a very short battery life expectancy in all market scenarios. The 1-segment model only arbitrages during large price deviations. In particular, the BES is never dispatched in the day-ahead because these market prices are very stable. 
	
	The BES achieves the largest profits in the 5-minute RTM because this market has the largest price fluctuations. The revenue from reserve is lower in the 5-minute RTM than the hourly RTM. This result shows that the proposed approach is able to switch the focus of BES operation from reserve to arbitrage when market price fluctuations become high. In the RTM, the BES collects a substantial portion of its profits from the provision of reserve, especially in the hourly RTM. A BES is more flexible than generators at providing reserves because it does not have a minimum stable generation, it can start immediately, and can remain idle until called. Therefore, the provision of reserve causes no cycle aging. In the hourly RTM, the provision reserve represents about 74\% of the market revenue and 90\% of the prorated profits for this BES.

	\section{Conclusion}
	
	This paper proposes a method for incorporating the cost of battery cycle aging in economic dispatch, market clearing or the development of bids and offers . This approach takes advantage of the flexibility that a battery can provide to the power system while ensuring that its operation remains profitable in a market environment. The cycle aging model closely approximates the actual electrochemical battery cycle aging mechanism, while being simple enough to be incorporated into market models such as economic dispatch. Based on simulations performed using a full year of actual market price data, we demonstrated the effectiveness and accuracy of the proposed model. These simulation results show that modeling battery degradation using the proposed model significantly improves the actual BES profitability and life expectancy.

	\appendix
	
	In this appendix we prove that the proposed piecewise linear model of the battery cycle aging cost is a close approximation of the benchmark rainflow-based battery cycle aging model, and that the accuracy of the model increases with the number of linearization segments. \emph{The proposed model produces the same aging cost as to the benchmark aging model for the same battery operation profile with an adequate number of linearization segments.} To prove this, we first explicitly characterize the cycle aging cost result calculated using the proposed model (Theorem~1). We then show that this cost approaches the benchmark result when the number of linearization segments approaches infinity (Theorem~2).

	We consider the operation of a battery over the period $\mathcal{T} = \{1,2,\dotsc,T\}$, the physical battery operation constraints are ($\forall t\in \mathcal{T}$)
	\begin{align}
	d_t &\leq D(1-v_t)
	\label{Eq:App_1}\\
	g_t &\leq Gv_t%\,,
	\label{Eq:App_2}\\
	e_t - e_{t-1} &= M(d_t\eta\up{ch}-g_t/\eta\up{dis})
	\label{Eq:App_3}% \\
	% e_1 &= e_0 + M(d_1\eta\up{ch}-g_1/\eta\up{dis})\,.
	% \label{Eq:App_4}
	\end{align}
	We denote $\mathbf{d}=\{d_1, d_2, \dotsc, d_T\}$ as the set of all battery charge powers, and  $\mathbf{g}=\{g_1, g_2, \dotsc, g_T\}$ as the set of all discharge powers. Hence, a set in the form of  $(\mathbf{d}, \mathbf{g})$ is sufficient to describe  the dispatch of a battery over $\mathcal{T}$. Let $\mathcal{P}(e_0)$ denote the set of all feasible battery dispatches that satisfy the physical battery operation constraints \eqref{Eq:App_1}--\eqref{Eq:App_3} given an battery initial energy level $e_0$. 
	
	Since we are only interested in characterizing the aging cost calculated by the proposed model for a certain battery operation profile, we will regard the battery operation profile as known variables in this proof. It is easy to see that once the dispatch profile $(\mathbf{d}, \mathbf{g})$ is determined, any battery dispatch problem that involves the proposed model with a linearization segment set $\mathcal{J}=\{1,2,\dotsc, J\}$, such as the one formulated in Section~\ref{Sec:opt}, can be reduced to the following problem if we neglect any operation prior to the operation interval $\mathcal{T}$
	\begin{align}
	\hat{\mathbf{p}} \in \mathrm{arg}&\min_{\mathbf{p}\in \mathbb{R}^+} \textstyle \sum_{t=1}^{T}\sum_{j=1}^{J}Mc\up{}_jp\up{dis}_{t,j}\,, 
	\label{Eq:Pro_obj}\\
	&\text{s.t. }\nonumber\\
	&d_t = \textstyle\sum_{j=1}^{J}p\up{ch}_{t,j}% /\eta\up{ch}
	\label{Eq:Pro_C1}\\
	&g_t = \textstyle\sum_{j=1}^{J}p\up{dis}_{t,j}% \eta\up{dis}
	\label{Eq:Pro_C2}\\
	&e\up{}_{t,j}-e\up{}_{t-1,j} = M(p\up{ch}_{t,j}\eta\up{ch}- p\up{dis}_{t,j}/\eta\up{dis})
	\label{Eq:Pro_C3}\\
	&0\leq e\up{}_{t,j} \leq \overline{e}\up{}_j
	\label{Eq:Pro_C4}\\
	&\textstyle \sum_{j=1}^{J} e\up{}_{0,j} = e_0
	\label{Eq:Pro_C5}
	\end{align}
	where $(\mathbf{d}, \mathbf{g})\in \mathcal{P}(e_0)$ is a feasible battery dispatch set, and $\mathbf{p} = \{p\up{ch}_{t,j}, p\up{dis}_{t,j}|t\in \mathcal{T}, j\in \mathcal{J}\}$ denotes a set of the battery charge and discharge powers for all segments during all dispatch intervals. Although the objective is still cost minimization,  the problem in \eqref{Eq:Pro_obj}--\eqref{Eq:Pro_C5} does not optimize battery dispatch, instead it simulates cycle operations $\mathbf{p}$ and calculates the cycle aging cost with respect to a dispatch profile $(\mathbf{d}, \mathbf{g})$. Hence, the evaluation criteria to this problem is its accuracy compared to the benchmark aging cost model.
	
	Let $\mathbf{c}=\{c_j|j\in \mathcal{J}\}$ denote a set of piecewise linear battery aging cost segments derived as in equation \eqref{ES:ca_pl}, so that $c_j$ is associated with the cycle depth range $[(j-1)/J, j/J)$ and $J=|\mathcal{J}|$ is the number of segments. We say that a battery has a \emph{convex} aging cost curve (i.e., non-decreasing marginal cycle aging cost) if a \emph{shallower} cycle depth segment (i.e., indexed with smaller $j$) is associated with a \emph{cheaper} marginal aging cost such that $c_1\leq c_2 \leq \dotsc \leq c_J$, and let $\mathcal{C}$ denote the set of all convex battery aging cost linearizations.

	\begin{theorem}
		Let $\hat{\mathbf{p}} = \{\hat{p}\up{ch}_{t,j}, \hat{p}\up{dis}_{t,j}|t\in \mathcal{T}, j\in \mathcal{J}\}$ and
		\begin{align}
		\hat{p}\up{ch}_{t,j} &= \textstyle \min\big[d_t-\sum_{\zeta=1}^{j-1}\hat{p}\up{ch}_{t,\zeta}, \;(\overline{e}_j- \hat{e}\up{}_{t-1,j})/(\eta\up{ch}M)\big] 
		\label{Eq:the1}\\
		\hat{p}\up{dis}_{t,j} &= \textstyle \min\big[g_t-\sum_{\zeta=1}^{j-1}\hat{p}\up{dis}_{t,\zeta}, \;\eta\up{dis}\hat{e}\up{}_{t-1,j}/M\big] 
		\label{Eq:the2}\\
		\hat{e}\up{}_{0,j} &= \textstyle \min\big[\overline{e}_j, \max(0, e_0-\sum_{\zeta=1}^{j-1}\hat{e}\up{}_{0,\zeta})\big]
		\label{Eq:the3}\\
		\hat{e}\up{}_{t,j} &= \hat{e}\up{}_{t-1,j} + M(p\up{ch}_{t,j}\eta\up{ch}- p\up{dis}_{t,j}/\eta\up{dis})\,.
		\label{Eq:the4}
		\end{align}
		Then $\hat{\mathbf{p}}$ is a minimizer of the problem \eqref{Eq:Pro_obj}--\eqref{Eq:Pro_C5} as long as the battery dispatch is feasible and the cycle aging cost curve is convex, i.e.,
		\begin{align}
		&\hat{\mathbf{p}} \in \mathrm{arg}\min_{\mathbf{p}\in \mathbb{R}^+} \text{ \eqref{Eq:Pro_obj}--\eqref{Eq:Pro_C5}}\,,\nonumber\\
		&\text{$\forall (\mathbf{d}, \mathbf{g})\in \mathcal{P}(e_0)$, $e_0 \in [E\up{min}, E\up{max}]$, $\mathbf{c} \in \mathcal{C}$}.
		\end{align}
	\end{theorem}
	
	\begin{proof}
		Equations \eqref{Eq:the1}--\eqref{Eq:the4} describe a battery operating policy over the proposed piecewise linear model. To calculate this policy, we start from \eqref{Eq:the3} which calculates the initial segment energy level from the battery initial SoC $e_0$. \eqref{Eq:the3} is evaluated in the order of $j=0,1,2,3,\dotsc, J$ such as (note that $\sum_{\zeta=1}^{0}\hat{e}\up{}_{0,\zeta} = 0$)
		\begin{align}
		\hat{e}\up{}_{0,1} &= \textstyle \min\big[\overline{e}_1, \max(0, e_0)\big]\nonumber\\
		\hat{e}\up{}_{0,2} &= \textstyle \min\big[\overline{e}_2, \max(0, e_0-\hat{e}\up{}_{0,1})\big]\nonumber\\
		\hat{e}\up{}_{0,3} &= \textstyle \min\big[\overline{e}_3, \max(0, e_0-\hat{e}\up{}_{0,1}-\hat{e}\up{}_{0,2})\big]\nonumber\\
		& \dots\,,\nonumber
		\end{align}
		so that energy in $e_0$ is first assigned to $\hat{e}_{0,1}$ which corresponds to the shallowest cycle depth range $[0, 1/J]$, the remaining energy is then assigned to the second shallowest segment $\hat{e}_{0,2}$, and the procedure repeats until all the energy in $e_0$ has been assigned. 
		
		We then calculate all battery segment charge power at $t=1$ in the order of $j=0,1,2,3,\dotsc, J$ as
		\begin{align}
		\hat{p}\up{ch}_{1,1} &= \textstyle \min\big[d_t, \;(\overline{e}_1- \hat{e}\up{}_{0,1})/(\eta\up{ch}M)\big] \nonumber\\
		\hat{p}\up{ch}_{1,2} &= \textstyle \min\big[d_t-\hat{p}\up{ch}_{1,1}, \;(\overline{e}_2 - \hat{e}\up{}_{0,2})/(\eta\up{ch}M)\big] \nonumber\\
		\hat{p}\up{ch}_{1,3} &= \textstyle \min\big[d_t-\hat{p}\up{ch}_{1,1}-\hat{p}\up{ch}_{1,2}, \;(\overline{e}_3- \hat{e}\up{}_{0,3})/(\eta\up{ch}M)\big]\nonumber\\
		&\dots\,,\nonumber
		\end{align}
		and the procedure is similar for segment discharge power $\hat{p}\up{dis}_{1,j}$. We calculate the segment energy level $\hat{e}_{1,j}$ at the end of $t=1$ using \eqref{Eq:the4}, and move the calculation to $t=2$. This procedure repeats until all values in $\hat{\mathbf{p}}$ have been calculated. Therefore in this policy, the battery always prioritizes energy in shallower segments for charge or discharge dispatch. For example, if the battery is required to discharge a certain amount of energy, it will first dispatch segment 1, then the remaining discharge requirement (if any) is dispatched from segment 2, then segment 3, etc.

		% Whenever the battery is instructed to discharge, this policy always discharges segments associated with a shallower cycle depth first, and then segments associated with a deeper cycle depth, as described in \eqref{Eq:the1}. Similarly, the minimizer always charges the cheaper blocks first when instructed to charge, as described in \eqref{Eq:the2}. All initial energy $e_0$ will be assigned to shallower segments first as in \eqref{Eq:the3}, while \eqref{Eq:the4} in this policy is identical to \eqref{Eq:Pro_C3}.
		
		Given this policy, this theorem stands if the battery cycle aging cost curve $\mathbf{c}$ is convex, i.e., $\mathbf{c}\in \mathcal{C}$, which means a shallower segment is associated with a cheaper marginal operating cost. Since the objective function \eqref{Eq:Pro_obj} is to minimize the battery aging cost and the problem involves no market price, then a minimizer for the problem \eqref{Eq:Pro_obj}--\eqref{Eq:Pro_C5} will give a cheaper segment a higher operation priority, which is equivalent to the policy described in \eqref{Eq:the1}--\eqref{Eq:the4}.
	\end{proof}
	
	Following Theorem~1, the cycle aging cost calculated by the proposed piecewise linear model $C\up{pwl}$ for a battery dispatch profile $(\mathbf{d}, \mathbf{g})$ can be written as a function of this profile and the linearization cost set as 
	\begin{align}
	C\up{pwl}(\mathbf{c}, \mathbf{d}, \mathbf{g}) = \textstyle\sum_{t=1}^{T}\sum_{j=1}^{J}Mc\up{}_j\hat{p}\up{dis}_{t,j}\,,
	\label{Eq:the_cost}
	\end{align}
	where $\hat{\mathbf{p}}$ is calculated as in \eqref{Eq:the1}--\eqref{Eq:the4}.

	Let $\Phi(\delta)$ be a convex battery cycle aging stress function, and  $\mathbf{c}(\Phi)$ be a set of piecewise linearizations of $\Phi(\delta)$ determined using the method described in equation \eqref{ES:ca_pl}. Let $|\mathbf{c}(\Phi)|$ denote the cardinality of $\mathbf{c}(\Phi)$, i.e. the number of segments in this piecewise linearization.  
	
	For a feasible battery dispatch profile $(\mathbf{d}, \mathbf{g}) \in \mathcal{P}(e_0)$, let $\Delta$ be the set of all full cycles identified from this operation profile using the rainflow method, $\Delta\up{dis}$ for all discharge half cycles, and $\Delta\up{ch}$ for all charge half cycles. The benchmark cycle aging cost $C\up{ben}$ resulting from $(\mathbf{d}, \mathbf{g})$ can be written as a function of the profile and the cycle aging function $\Phi$ (recall that a full cycle has symmetric depths for charge and discharge)
	\begin{align}
	C\up{ben}(\Phi, \mathbf{d}, \mathbf{g}) = \textstyle R\sum_{i=1}^{|\Delta|}\Phi(\delta_i) + R\sum_{i=1}^{|\Delta\up{dis}|}\Phi(\delta\up{dis}_i)\,.
	\end{align}
	
	\begin{theorem}
		When the number of linearization segments approaches infinity, the proposed piecewise linear cost model yields the same result as the benchmark rainflow-based cost model:
		\begin{align}
		\lim_{|\mathbf{c}(\Phi)|\to \infty} C\up{pwl}\big(\mathbf{c}(\Phi), \mathbf{d}, \mathbf{g}\big) =  C\up{ben}\big(\Phi, \mathbf{d}, \mathbf{g}\big)\,.
		\label{Eq:the_2}
		\end{align}
	\end{theorem}
	
	\begin{proof}
		First we rewrite equation \eqref{Eq:the_cost} as
		\begin{align}
		\textstyle\sum_{j=1}^{J}c\up{}_j\sum_{t=1}^{T}M\hat{p}\up{dis}_{t,j} = \sum_{j=1}^{J}c\up{}_j\Theta_j\,,
		\end{align}
		where $\Theta_j = \sum_{t=1}^{T}M\hat{p}\up{dis}_{t,j}$ is the total amount of energy discharged at a cycle depth range between $(j-1)/J$ and $j/J$. Once the number of segments $|\mathbf{c}(\Phi)| = J$ approaches infinity, we can rewrite $\Theta_j$ into a function $\Theta(\delta)$ indicating the energy discharged at a specific cycle depth $\delta$, where $\delta\in [0 \; 1]$. With an infinite number of segments, we substitute \eqref{ES:inc} in and rewrite the cycle aging function in \eqref{Eq:the_cost} in a continuous form
		\begin{align}
		C\up{pwl}(\Phi, \mathbf{d}, \mathbf{g}) = \int_{0}^{1}\frac{R}{\eta\up{dis}E\up{rate}}\Theta(\delta)\frac{d\Phi(\delta)}{d\delta}d\delta\,.
		\end{align}
		We define a new function $N\up{dis}_T(\delta)$ the number of discharge cycles of depths equal or greater than $\delta$ during the operation period from $t=0$ to $t=T$, accounting all discharge half cycles and the discharge stage of all full cycles. $N\up{dis}_T(\delta)$ can be calculated by normalizing $\Theta(\delta)$ with the discharge efficiency and the energy rating of the battery
		\begin{align}
		N\up{dis}_T(\delta) = \frac{1}{\eta\up{dis}E\up{rate}}\Theta(\delta)\,,
		\end{align}
		recall that $\Theta(\delta)$ is the amount of energy discharged from the cycle depth $\delta$. This relationship is proved in Lemma 1 after this theorem. 
		% \textcolor{blue}{For example, consider an operation that consists only a cycle of depth $x$ and discharges $xE\up{rate}/\eta\up{dis}$ amount of energy, which is described by the following equation
		% \begin{align}
		%     \int_{0}^{1}\Theta(\delta)d \delta = x\frac{E\up{rate}}{\eta\up{dis}}\,,
		% \end{align}
		% we substitute $N(\delta)$ into the previous equation 
		% \begin{align}
		%     \int_{0}^{1} \frac{E\up{rate}}{\eta\up{dis}}N(\delta)d \delta &= x\frac{E\up{rate}}{\eta\up{dis}}
		% \end{align}
		% hence
		% \begin{align}
		%     \int_{0}^{1} N(\delta)d \delta &= x
		% \end{align}
		% }

		% Once we normalize $\Theta(\delta)$ with the discharge efficiency and the energy rating of the battery, it is easy to see that $N(\delta)$ means the number of times the battery is discharged to a cycle depth $\delta$, accounting for all full cycles and discharge half cycles. 
		Now the proposed cost function becomes
		\begin{align}
		C\up{pwl}(\Phi, \mathbf{d}, \mathbf{g}) = R\int_{0}^{1}\frac{d\Phi(\delta)}{d\delta}N\up{dis}_T(\delta)d\delta\,,
		\label{eq:the2_p1}
		\end{align}
		which is a standard formulation for calculating rainflow fatigue damage~\cite{rychlik1996extremes}, and the function $N\up{dis}_T(\delta)$ is an alternative way of representing a rainflow cycle counting result. We substitute \eqref{Eq:lemma1} from Lemma 1 into \eqref{eq:the2_p1}
		\begin{align}
		&C\up{pwl}(\Phi, \mathbf{d}, \mathbf{g}) \nonumber\\
		& = R\int_{0}^{1}\frac{d\Phi(\delta)}{d\delta}\Bigg( \sum_{i=1}^{|\Delta|}1_{[\delta \leq \delta_i]} + \sum_{i=1}^{|\Delta\up{dis}|}1_{[\delta \leq \delta\up{dis}_i]} \Bigg)d\delta\nonumber\\
		& = R\sum_{i=1}^{|\Delta|}\int_{0}^{1} \frac{d\Phi(\delta)}{d\delta} 1_{[\delta \leq \delta_i]} d\delta + R\sum_{i=1}^{|\Delta\up{dis}|}\int_{0}^{1} \frac{d\Phi(\delta)}{d\delta}1_{[\delta \leq \delta\up{dis}_i]} d\delta\nonumber\\
		&= R\sum_{i=1}^{|\Delta|} \Phi(\delta_i) + R\sum_{i=1}^{|\Delta\up{dis}|} \Phi(\delta\up{dis}_i)\nonumber\\
		&= C\up{ben}(\Phi, \mathbf{d}, \mathbf{g})\,,
		\end{align}
		then it is trivial to see that this theorem stands if the proposed model yields the same counting result $N\up{dis}_T(\delta)$ as the rainflow algorithm. This relationship is proved in Lemma~1.
		
		% from equation \eqref{ES:ca_pl} it is trivial to see that for cycles of the same depth, the proposed piecewise linear model should yield the same cost as the benchmark model when the number of segments approaches infinity, since \eqref{Eq:the_cost} becomes the same as 
	\end{proof}

	\begin{lemma}
		We assume that the proposed model has an infinite number of segments, then $N\up{dis}_T(\delta)$, as defined in Theorem 2, is the number of discharge cycles of depths equal or greater than $\delta$ during the operation period from $t=0$ to $t=T$, accounting all discharge half cycles and the discharge stage of all full cycles, hence
		\begin{align}
		N\up{dis}_T(\delta) &= \Theta(\delta)/(\eta\up{dis}E\up{rate}) \label{Eq:lemma1_2}\\
		&= \textstyle\sum_{i=1}^{|\Delta|}1_{[\delta \leq \delta_i]} + \sum_{i=1}^{|\Delta\up{dis}|}1_{[\delta \leq \delta\up{dis}_i]}\,,
		\label{Eq:lemma1}
		\end{align}
		where $1_{[x]}$ has a value of one if $x$ is true, and zero otherwise.
	\end{lemma}
	\begin{proof}
		\eqref{Eq:lemma1_2} defines $N\up{dis}_T(\delta)$ as the number fo times that energy is discharged from the cycle depth $\delta$, while \eqref{Eq:lemma1} means the number of cycles with depths at least $\delta$. Therefore in this lemma we prove that these two definitions are equivalent, hence the proposed model has the same cycle counting result as the rainflow method.
		
		Let $N\up{dis}_t(\delta)$ be the number of times energy is discharged from the depth $\delta$ during the operation period $[0,t]$, accounting all discharge half cycles and the discharge stage of all full cycles. Similarly, define $N\up{ch}_t(\delta)$ accounting all charge half cycles and the charge stage of all full cycles. 
		\begin{figure}[!htb]
			\centering
			\includegraphics[trim = 10mm 0mm 10mm 0mm, clip, width = .95\columnwidth]{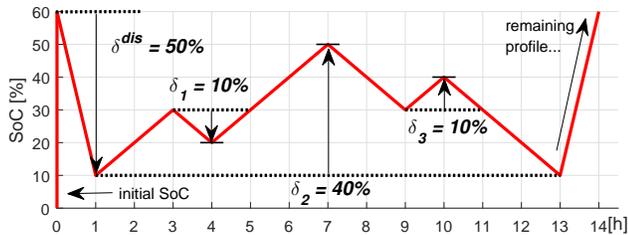}
			\caption{Cycle counting example.}
			\label{Fig:lemma1}
		\end{figure}
		Because we assume charge dispatches cause no aging cost, we can alternatively model battery initial energy level $e_0$ as an empty battery being charged to $e_0$ at the beginning of operation (such as in Fig.~\ref{Fig:lemma1}), hence at $t=0$ we have
		\begin{align}
		N\up{ch}_0(\delta) = \begin{cases} 1 & \delta \leq e_0 \\ 0 & \delta > e_0 \end{cases}\,,\quad N\up{dis}_0(\delta) = 0\,.
		\end{align}
		Now assume at time $t_1$ the battery is switched from charging to discharging, and eventually resulted in a cycle of depth $x$ that ends at $t_2$, regardless whether it is a half cycle or a full cycle. We also assume that there is no other cycles occuring from $t_1$ tp $t_2$, since in the rainflow method  The battery must have been previously charged at least $\delta$ depth worth of energy since we now assume the battery starts from empty. Therefore according to Theorem 1, segments in the range $[0,x]$ must be full at $t_1$, hence
		\begin{align}
		N\up{ch}_{t_1}(\delta) - N\up{dis}_{t_1}(\delta) = 1 \quad \forall \delta \leq x\,,
		\end{align}
		which is a sufficient condition for all discharge energy in this cycle being dispatched from segments in the depth range $[0,x]$, according to Theorem 1. After performing this cycle, all and only segments within the range $[0,x]$ are discharged one more time, in other words, all and only cycle depths in the range $[0,x]$ have one more count at end of this cycle $t_2$ compared to $t_1$ when the discharge begins, hence
		\begin{align}
		N\up{dis}_{t_2}(\delta)-N\up{dis}_{t_1}(\delta) = 1_{[\delta\leq x]}.
		\end{align}
		Therefore the proposed model has the same counting result as to the rainflow method for any cycles, which proves this lemma.

	\end{proof}

	\subsection{Numerical example}

	\begin{figure}[!htb]
		\centering
		\includegraphics[trim = 10mm 0mm 10mm 0mm, clip, width = .95\columnwidth]{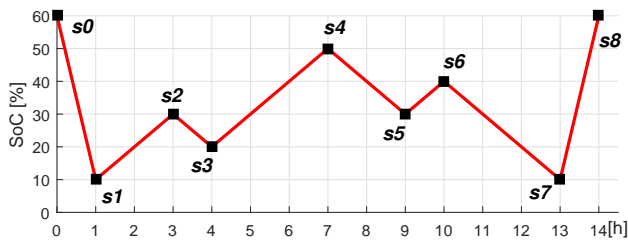}
		\caption{An example of SoC profile.}
		\label{Fig:app}
	\end{figure}
	
	\begin{table}[!hbt]
		\begin{center}
			\centering
			% \footnotesize
			\caption{Battery Operation Example.}
			\label{tab:es_example}
			% \begin{footnotesize}
			\begin{tabular}{r  c  c c  c}
				\hline
				\hline
				t & SoC & energy segments & discharge power & cost 	\Tstrut \\
				&     & $[\mathbf{e}_t]$ & $[\mathbf{p}\up{dis}_t]$ & $C_t$ \Bstrut\\
				\hline
				- & - & $\to$ deeper depth $\to$ & $\to$ deeper depth $\to$ & -\Tstrut\Bstrut\\
				\hline
				0  & 60 & 1,1,1,1,1,1,0,0,0,0 & 0,0,0,0,0,0,0,0,0,0 & 0	\Tstrut\Bstrut\\
				\hline
				1  & 10 & 0,0,0,0,0,1,0,0,0,0 & 1,1,1,1,1,0,0,0,0,0 & 25	\Tstrut\Bstrut\\
				\hline
				2  & 20 & 1,0,0,0,0,1,0,0,0,0 & 1,1,1,1,1,0,0,0,0,0 & 0	\Tstrut\Bstrut\\
				\hline
				3  & 30 & 1,1,0,0,0,1,0,0,0,0 & 0,0,0,0,0,0,0,0,0,0 & 0	\Tstrut\Bstrut\\
				\hline
				4  & 20 & 0,1,0,0,0,1,0,0,0,0 & 1,0,0,0,0,0,0,0,0,0 & 1	\Tstrut\Bstrut\\
				\hline
				5  & 30 & 1,1,0,0,0,1,0,0,0,0 & 0,0,0,0,0,0,0,0,0,0 & 0	\Tstrut\Bstrut\\
				\hline
				6  & 40 & 1,1,1,0,0,1,0,0,0,0 & 0,0,0,0,0,0,0,0,0,0 & 0	\Tstrut\Bstrut\\
				\hline
				7  & 50 & 1,1,1,1,0,1,0,0,0,0 & 0,0,0,0,0,0,0,0,0,0 & 0	\Tstrut\Bstrut\\
				\hline
				8  & 40 & 0,1,1,1,0,1,0,0,0,0 & 1,0,0,0,0,0,0,0,0,0 & 1	\Tstrut\Bstrut\\
				\hline
				9  & 30 & 0,0,1,1,0,1,0,0,0,0 & 0,1,0,0,0,0,0,0,0,0 & 3	\Tstrut\Bstrut\\
				\hline
				10 & 40 & 1,0,1,1,0,1,0,0,0,0 & 0,0,0,0,0,0,0,0,0,0 & 0	\Tstrut\Bstrut\\
				\hline
				11 & 30 & 0,0,1,1,0,1,0,0,0,0 & 1,0,0,0,0,0,0,0,0,0 & 1	\Tstrut\Bstrut\\
				\hline
				12 & 20 & 0,0,0,1,0,1,0,0,0,0 & 0,0,1,0,0,0,0,0,0,0 & 5	\Tstrut\Bstrut\\
				\hline
				13 & 10 & 0,0,0,0,0,1,0,0,0,0 & 0,0,0,1,0,0,0,0,0,0 & 7	\Tstrut\Bstrut\\
				\hline
				14 & 60 & 1,1,1,1,1,1,0,0,0,0 & 0,0,0,0,0,0,0,0,0,0 & 0	\Tstrut\Bstrut\\
				\hline
				all & - & - & - & 43	\Tstrut\Bstrut\\
				\hline
				\hline						
			\end{tabular}
			%	\end{footnotesize}
		\end{center}
		\vspace{-.5cm}
	\end{table}
	
	We include a step-by-step example to illustrate how the proposed model is a close approximation of the benchmark rainflow cost model using the battery operation profile shown in Fig.~\ref{Fig:app}. To simplify this example, we assume a perfect efficiency of $1$ and that the cycle aging cost function is $100\delta^2$. We consider 10 linearization segments, with each segment representing a 10\% cycle depth range. The proposed model therefore has the following cycle aging cost curve
	\begin{align}
	\mathbf{c}=\{1, 3, 5, 7, 9, 11, 13, 15, 17, 19\}.
	\end{align}
	According to the rainflow method demostrated in Fig.~\ref{Fig:rf} , this example profile has the following cycle counting results
	\begin{itemize}
		\item Two full cycles of depth 10\%, each costs 1
		\item One full cycle of depth 40\% that costs 16
		\item One discharge half cycle of depth 50\% that costs 25
		\item One charge half cycle that costs zero,
	\end{itemize}
	hence the total aging cost identified by the benchmark rainflow-based model is 43. 
	
	We implement this operation profile using the policy in Theorem~1 and record the marginal cost during each time interval. The results are shown in Table~\ref{tab:es_example}. In this table, the first two columns are the time step and SoC. The third column shows the energy level of each linearization segment represented in a vector from $\mathbf{e}_t$. $\mathbf{e}_t$ is a $10\times 1$ vector, and its energy level segments are sorted from shallower to deeper depths. Segment energy levels are normalized so that one means the segment is full, and zero means the segment is empty. The fourth column shows how much energy is discharged from each segment during a time interval, represented by a discharge power vector $\mathbf{p}\up{dis}_t$ and is calculated as (the discharge efficiency is 1)
	\begin{align}
	\mathbf{p}\up{dis}_t = [\mathbf{e}_{t-1}-\mathbf{e}_t]^+\,.
	\end{align}
	The last column shows the operating cost that arises from each time interval, which is calculated as
	\begin{align}
	C_t = \mathbf{c}\mathbf{p}\up{dis}_t\,.
	\end{align}
	
	This example profile results in the same cost of 43 in both the proposed model and the benchmark model, as proved in Theorem~2.
	
	\bibliographystyle{IEEEtran}	% (uses file "plain.bst")
	\bibliography{IEEEabrv,literature}		% expects file "myrefs.bib"
	
	\begin{IEEEbiographynophoto}{Bolun Xu}
		(S'14) received B.S. degrees in Electrical and Computer Engineering
		from Shanghai Jiaotong
		University, Shanghai, China in 2011, and the M.Sc degree in Electrical
		Engineering from Swiss Federal Institute of Technology, Zurich, Switzerland
		in 2014.
		
		He is currently pursuing the Ph.D. degree in Electrical Engineering at the
		University of Washington, Seattle, WA, USA. His research interests include
		energy storage, power system operations, and power system economics.
	\end{IEEEbiographynophoto}
	
	\begin{IEEEbiographynophoto}{Jinye Zhao}
		(M'11) received the B.S. degree from East China Normal University,
		Shanghai, China, in 2002 and the M.S. degree in mathematics from National
		University of Singapore in 2004. She received the M.E. degree in operations
		research and statistics and the Ph.D. degree in mathematics from Rensselaer
		Polytechnic Institute, Troy, NY, in 2007.
		
		She is a lead analyst at ISO New England, Holyoke, MA. Her main interests
		are game theory, mathematical programming, and electricity market modeling.
	\end{IEEEbiographynophoto}
	
	\begin{IEEEbiographynophoto}{Tongxin Zheng}
		(SM'08) received the B.S. degree in electrical engineering
		from North China University of Electric Power, Baoding, China, in 1993, the
		M.S. degree in electrical engineering from Tsinghua University, Beijing, China,
		in 1996, and the Ph.D. degree in electrical engineering from Clemson
		University, Clemson, SC, USA, in 1999. 
		
		Currently, he is a Technical Manager
		with the ISO New England, Holyoke, MA, USA. His main interests are power
		system optimization and electricity market design.
		
	\end{IEEEbiographynophoto}

	\begin{IEEEbiographynophoto}{Eugene Litvinov}
		(SM'06-F'13) received the B.S. and M.S. degrees from the
		Technical University, Kiev, Ukraine, and the Ph.D. degree from Urals
		Polytechnic Institute, Sverdlovsk, Russia. 
		
		Currently, he is the Chief
		Technologist at the ISO New England, Holyoke, MA. His main interests
		include power system market-clearing models, system security, computer
		applications in power systems, and information technology. 
	\end{IEEEbiographynophoto}

	\begin{IEEEbiographynophoto}{Daniel S. Kirschen}
		(M'86-SM'91-F'07) received his electrical and mechanical engineering degree from the Universite Libre de Bruxelles, Brussels, Belgium, in 1979 and his M.S. and Ph.D. degrees from the University of Wisconsin, Madison, WI, USA, in 1980, and 1985, respectively.
		
		He is currently the Donald W. and Ruth Mary Close Professor of Electrical Engineering at the University of Washington, Seattle, WA, USA. His research interests include smart grids, the integration of renewable energy sources in the grid, power system economics, and power system security.
	\end{IEEEbiographynophoto}
	
\end{document}